\documentclass[12pt]{amsart}
\usepackage{amsmath,amssymb,amsbsy,amsfonts,amsthm,latexsym,
                        amsopn,amstext,amsxtra,euscript,amscd,mathrsfs,color,bm}
                        
\usepackage{float} 
\usepackage[english]{babel}
\usepackage{url}
\usepackage[colorlinks,linkcolor=blue,anchorcolor=blue,citecolor=blue,backref=page]{hyperref}

\usepackage[margin=3.5cm]{geometry}
\usepackage{color}
\usepackage[norefs,nocites]{refcheck}

\renewcommand*{\backref}[1]{}
\renewcommand*{\backrefalt}[4]{%
    \ifcase #1 (Not cited.)%
    \or        (p.\,#2)%
    \else      (pp.\,#2)%
    \fi}

\begin{document}

\newtheorem{theorem}{Theorem}
\newtheorem{lemma}[theorem]{Lemma}
\newtheorem{example}[theorem]{Example}
\newtheorem{algol}{Algorithm}
\newtheorem{corollary}[theorem]{Corollary}
\newtheorem{prop}[theorem]{Proposition}
\newtheorem{definition}[theorem]{Definition}
\newtheorem{question}[theorem]{Question}
\newtheorem{problem}[theorem]{Problem}
\newtheorem{remark}[theorem]{Remark}
\newtheorem{conjecture}[theorem]{Conjecture}

\newcommand{\commM}[1]{\marginpar{%
\begin{color}{red}
\vskip-\baselineskip 
\raggedright\footnotesize
\itshape\hrule \smallskip M: #1\par\smallskip\hrule\end{color}}}

\newcommand{\commA}[1]{\marginpar{%
\begin{color}{blue}
\vskip-\baselineskip 
\raggedright\footnotesize
\itshape\hrule \smallskip A: #1\par\smallskip\hrule\end{color}}}
\def\xxx{\vskip5pt\hrule\vskip5pt}


\def\cA{{\mathcal A}}
\def\cB{{\mathcal B}}
\def\cC{{\mathcal C}}
\def\cD{{\mathcal D}}
\def\cE{{\mathcal E}}
\def\cF{{\mathcal F}}
\def\cG{{\mathcal G}}
\def\cH{{\mathcal H}}
\def\cI{{\mathcal I}}
\def\cJ{{\mathcal J}}
\def\cK{{\mathcal K}}
\def\cL{{\mathcal L}}
\def\cM{{\mathcal M}}
\def\cN{{\mathcal N}}
\def\cO{{\mathcal O}}
\def\cP{{\mathcal P}}
\def\cQ{{\mathcal Q}}
\def\cR{{\mathcal R}}
\def\cS{{\mathcal S}}
\def\cT{{\mathcal T}}
\def\cU{{\mathcal U}}
\def\cV{{\mathcal V}}
\def\cW{{\mathcal W}}
\def\cX{{\mathcal X}}
\def\cY{{\mathcal Y}}
\def\cZ{{\mathcal Z}}

\def\C{\mathbb{C}}
\def\F{\mathbb{F}}
\def\K{\mathbb{K}}
\def\G{\mathbb{G}}
\def\Z{\mathbb{Z}}
\def\R{\mathbb{R}}
\def\Q{\mathbb{Q}}
\def\N{\mathbb{N}}
\def\M{\textsf{M}}
\def\U{\mathbb{U}}
\def\P{\mathbb{P}}
\def\A{\mathbb{A}}
\def\p{\mathfrak{p}}
\def\n{\mathfrak{n}}
\def\X{\mathcal{X}}
\def\x{\textrm{\bf x}}
\def\w{\textrm{\bf w}}
\def\ovQ{\overline{\Q}}
\def\rank#1{\mathrm{rank}#1}
\def\wf{\widetilde{f}}
\def\wg{\widetilde{g}}
\def\comp{\hskip -2.5pt \circ  \hskip -2.5pt}
\def\({\left(}
\def\){\right)}
\def\[{\left[}
\def\]{\right]}
\def\<{\langle}
\def\>{\rangle}

\def\gen#1{{\left\langle#1\right\rangle}}
\def\genp#1{{\left\langle#1\right\rangle}_p}
\def\genPs{{\left\langle P_1, \ldots, P_s\right\rangle}}
\def\genPsp{{\left\langle P_1, \ldots, P_s\right\rangle}_p}

\def\e{e}

\def\eq{\e_q}
\def\fh{{\mathfrak h}}

\def\lcm{{\mathrm{lcm}}\,}

\def\l({\left(}
\def\r){\right)}
\def\fl#1{\left\lfloor#1\right\rfloor}
\def\rf#1{\left\lceil#1\right\rceil}
\def\mand{\qquad\mbox{and}\qquad}

\def\jt{\tilde\jmath}
\def\ellmax{\ell_{\rm max}}
\def\llog{\log\log}

\def\m{{\rm m}}
\def\ch{\hat{h}}
\def\GL{{\rm GL}}
\def\Orb{\mathrm{Orb}}
\def\Per{\mathrm{Per}}
\def\Preper{\mathrm{Preper}}
\def \S{\mathcal{S}}
\def\vec#1{\mathbf{#1}}
\def\ov#1{{\overline{#1}}}
\def\Gal{{\rm Gal}}

\newcommand{\bfalpha}{{\boldsymbol{\alpha}}}
\newcommand{\bfomega}{{\boldsymbol{\omega}}}

\newcommand{\Ch}{{\operatorname{Ch}}}
\newcommand{\Elim}{{\operatorname{Elim}}}
\newcommand{\proj}{{\operatorname{proj}}}
\newcommand{\h}{{\operatorname{h}}}

\newcommand{\hh}{\mathrm{h}}
\newcommand{\aff}{\mathrm{aff}}
\newcommand{\Spec}{{\operatorname{Spec}}}
\newcommand{\Res}{{\operatorname{Res}}}

\numberwithin{equation}{section}
\numberwithin{theorem}{section}

\def\house#1{{%
    \setbox0=\hbox{$#1$}
    \vrule height \dimexpr\ht0+1.4pt width .5pt depth \dimexpr\dp0+.8pt\relax
    \vrule height \dimexpr\ht0+1.4pt width \dimexpr\wd0+2pt depth \dimexpr-\ht0-1pt\relax
    \llap{$#1$\kern1pt}
    \vrule height \dimexpr\ht0+1.4pt width .5pt depth \dimexpr\dp0+.8pt\relax}}

\newcommand{\Address}{{
\bigskip
\footnotesize
\textsc{Centro di Ricerca Matematica Ennio De Giorgi, Scuola Normale Superiore, Pisa, 56126, Italy}\par\nopagebreak
\textit{E-mail address:} \texttt{marley.young@sns.it}
}}

\title[]
{On multiplicatively dependent vectors of polynomial values}

\author[M. Young] {Marley Young}

\subjclass[2020]{11N25, 11C08, 11R04}

\begin{abstract}
Given polynomials $f_1,\ldots,f_n$ in $m$ variables with integral coefficients, we give upper bounds for the number of integral $m$-tuples $\bm u_1,\ldots, \bm u_n$ of bounded height such that $f_1(\bm u_1), \ldots, f_n(\bm u_n)$ are multiplicatively dependent. We also prove, under certain conditions, a finiteness result for $\bm u \in \Z^m$ with relatively prime entries such that $f_1(\bm u),\ldots,f_n(\bm u)$ are multiplicatively dependent.
\end{abstract}

\maketitle

\section{Introduction}

We say that $n$ non-zero complex numbers $\nu_1, \ldots, \nu_n$ are \emph{multiplicatively dependent} if if there is a non-zero vector $(k_1,\ldots,k_n) \in \Z^n$ for which
\begin{equation} \label{eq:MultDep}
\nu_1^{k_1} \cdots \nu_n^{k_n} = 1.
\end{equation}
Otherwise we say they are \emph{multiplicatively independent}. Consequently, a point in the complex space $\C^n$ is called \emph{multiplicatively dependent} if its coordinates are all non-zero and multiplicatively dependent.

Multiplicative dependence of algebraic numbers is a deep and long-studied topic, see for example \cite{BBGMOS2,BOSS,BMZ,OSSZ1,OSSZ,PSSS,PL}.

In \cite{BMZ}, Bombieri, Masser and Zannier studied intersections of geometrically irreducible algebraic curves with proper algebraic subgroups of the multiplicative group $\G_{\mathrm{m}}^n$. Since such subgroups of $\G_{\mathrm{m}}^n$ are defined by finite sets of equations of the form $X^{k_1} \cdots X^{k_n}=1$ (see \cite[Corollary~3.2.15]{BG}), the paper \cite{BMZ} really concerns multiplicative dependence of points on a curve $\cX \subset \G_{\mathrm{m}}^n$. If $\cX$ is not contained in any translate of a proper algebraic subgroup of $\G_{\mathrm{m}}^n$, then the multiplicatively dependent points on $\cX(\overline \Q)$ form a set of bounded Weil height \cite[Theorem~1]{BMZ}.

Additionally, in \cite{OSSZ1}, for the maximal abelian extension $K^{\mathrm{ab}}$ of a number field $K$, the authors established the structure of multiplicatively dependent points on $\cX(K^{\mathrm{ab}})$. This in particular led to \emph{finiteness} results for such points, with the genus zero case, which equivalently concerns multiplicatively dependent values of rational functions, treated separately in \cite{OSSZ}.

\begin{theorem}\cite[Theorem~4.2]{OSSZ} \label{thm:OSSZ}
Let $F=(f_1,\ldots,f_n) \in K(X)^n$, whose components cannot multiplicatively generate a power of a linear fractional transformation. Then there are only finitely many elements $\alpha \in K^{\mathrm{ab}}$ such that $F(\alpha)$ is multiplicatively dependent.
\end{theorem}

This has implications for results in \emph{arithmetic dynamics} (for example \cite[Theorems~4.5~and~4.11]{OSSZ}), and these results have since been extended to hold modulo (approximate) finitely generated groups, see \cite{BBGMOS2,BOSS}. Mello \cite{Mello}, has proved analogues of the results of \cite{OSSZ1} in higher dimensions. It would be of interest to see if there is also a non-density result for rational functions in several variables analogous to Theorem~\ref{thm:OSSZ}. In a slightly different direction, we prove a finiteness result in the case of certain polynomials at integer points (see Proposition~\ref{prop:gcdFinite} below).

One can also consider questions of \emph{arithmetic statistics} in the context of multiplicative dependence. In \cite{PSSS}, the authors prove several asymptotic formulas for the number of multiplicatively dependent vectors of algebraic numbers or algebraic integers of fixed degree (or lying in a fixed number field) and bounded height. Bounds for the number of multiplicatively dependent matrices with integer entries of bounded size have also recently been obtained \cite{OS}.

As the main focus of this paper, we will in a similar vein count multiplicatively dependent polynomial or rational function values under certain constraints. For a subset $S \subseteq \C^m$, and an $n$-tuple of rational functions $F = (f_1, \ldots, f_n) \in \C(X_1,\ldots,X_m)^n$, denote by $N_{F}(S)$ the number of $n$-tuples $(\bm u_1,\ldots, \bm u_n) \in S^n$ such that $(f_1(\bm{u}_1),\ldots,f_n(\bm{u}_n))$ is multiplicatively dependent. Moreover, denote by $N_F^*(S)$ the number of $\bm u \in S$ such that $(f_1(\bm u),\ldots,f_n(\bm u))$ is multiplicatively dependent. We will look particularly at the case where $F$ consists of polynomials with integral coefficients and $S = [-H,H]^m = \{x \in \Z : |x| \leq H \}^m$ for some $H > 0$ (or more generally algebraic integers of bounded height in a number field). We note in the following examples that $N_F([-H,H])$ can range from zero to quite large depending on the choice of $F$.

\begin{example}
Let $p_i$ denote the $i$-th prime, let $n \geq 2$ and let $M=p_n$. Let $g(X) := X(X-1) \cdots (X-M)$, for $1 \leq i \leq n$, define $f_i(X) := g(X)+ p_i$, and let $F=(f_1,\ldots,f_n)$. By construction, for any $u \in \Z$ and $1 \leq j \leq n$, $g(X)$ has a factor of $X-(u\mod{p_j})$, and so $g(u) \equiv 0 \pmod{p_j}$. Thus, for any $(u_1,\ldots,u_n) \in \Z^n$ and $1 \leq i \neq j \leq n$, we have $p_i \mid f_i(u_i)$ and $p_i \nmid f_j(u_j)$. That is, $(f_1(u_1),\ldots,f_n(u_n))$ is never multiplicatively dependent, and so $N_F(\Z)=0$.
\end{example}

\begin{example}
Let $F=(f_1,\ldots,f_n) \in \Z[X_1,\ldots,X_m]^n$ be such that $f_i$ and $f_j$ are multiplicatively dependent for some $i \neq j$. Then $(f_1(\bm{u}_1),\ldots,f_n(\bm{u}_j))$ is multiplicatively dependent for any $\bm{u}_1,\ldots,\bm{u}_n \in \Z^m$ with $\bm{u}_i=\bm{u}_j$. Hence for $H \geq 0$, $N_F([-H,H]^m) \gg H^{m(n-1)}$. If more generally we have that for some $i \neq j$ there exist $\bm{u}_i, \bm{u}_j \in \Z^m$ such that $f_i(\bm{u}_i)$ and $f_j(\bm{u}_j)$ are multiplicatively dependent, then fixing such $\bm{u}_i$, $\bm{u}_j$ gives $N_F([-H,H])^m \gg H^{m(n-2)}$.
\end{example}

\begin{example}
Take $f_i(X)=X$ for $i=1,\ldots,n$ and let $F=(f_1,\ldots,f_n) \in \Z[X]^n$. Then $N_F([-H,H])$ is just the number of multiplicatively dependent integral vectors whose coordinates have height at most $H$. Thus by \cite[Equation~(1.16)]{PSSS},
$$
N_F([-H,H]) = n(n+1)(2H)^{n-1} + O \left( H^{n-2} \exp(c_0(n) \log H/\log \log H) \right),
$$
for some positive constant $c_0(n)$.
\end{example}

\subsection{Conventions and notation} We use the Landau symbols $O$ and $o$ and the Vinogradov symbol $\ll$. Above, and throughout the rest of the paper, the assertions $U=O(V)$ and $U \ll V$ are both equivalent to the inequality $|U| \leq cV$ with some positive constant $c$, while $U=o(V)$ means that $U/V \to 0$.

For $\bm{\nu} \in (\C^\times)^n$, we define the \emph{multiplicative rank}, $s$, of $\bm{\nu}$, in the following way. If $\bm{\nu}$ has a coordinate which is a root of unity, we put $s=0$; otherwise let $s$ be the largest integer with $1 \leq s \leq n$ for which any $s$ coordinates of $\bm{\nu}$ form a multiplicatively independent vector. Note that $0 \leq s \leq n-1$ whenever $\bm{\nu}$ is multiplicatively dependent. For $0 \leq s < n$, we define $N_{F,s}(S) \subset N_{F}(S)$ to consist of those elements $(\bm u_1,\ldots, \bm u_n)$ of $N_{F}(S)$ such that $(f_1(\bm u_1), \ldots, f_n(\bm u_n))$ is multiplicatively dependent of rank $s$. Then
\begin{equation} \label{eq:rankDecomp}
N_{F}(S) = N_{F,0}(S) + \ldots + N_{F,n-1}(S).
\end{equation}

For any algebraic number $\alpha$, let
$$
p(X)=a_d X^d + \cdots + a_1 X + a_0
$$
be the minimal polynomial of $\alpha$ over the integers, factored over $\C$ as
$$
p(x)=a_d(x-\alpha_1) \cdots (x-\alpha_d).
$$
The \emph{naive height} $\mathrm{H}_0(\alpha)$ of $\alpha$ is given by
$$
\mathrm{H}_0(\alpha) = \max \{ |a_d|, \ldots, |a_1|,|a_0| \},
$$
and the \emph{absolute Weil height} (or simply \emph{height}), $\mathrm{H}(\alpha)$ of $\alpha$ is defined by
$$
\mathrm{H}(\alpha) = \left( a_d \prod_{i=1}^d \max \{ 1, |\alpha_i| \} \right)^{1/d}.
$$
We will always assume that a bound $H > 0$ on the height of points is sufficiently large such that the logarithmic expressions $\log H$ and $\log \log  H$ are well-defined.

Let $K$ be a global field and let $f \in K[X_1,\ldots,X_m]$. For an integer $k \geq 1$, we say that $f$ is \emph{$k$-irreducible over $K$} if $f$ does not have any factors of degree at most $k$ over $K$. In Section~\ref{sec:Stat}, we will make use of dimension growth bounds for the number of integral points of bounded height on an affine hypersurface defined by a polynomial $f$ of degree $d \geq 2$. Such results typically assume that the degree $d$ part $f_d$ of $f$ is absolutely irreducible, but this requirement has recently been loosened by Cluckers et al. \cite{CDHNV}, who only assume $1$-irreducibility together with the following condition, which means that $f$ depends non-trivially on at least three variables in any affine coordinate system over the base field $K$.

\begin{definition} \label{def:NCC}
Let $K$ be a global field, let $f \in K[X_1,\ldots,X_m]$ with $m \geq 3$, and let $X=V(f)$ be the affine hypersurface cut out by $f$. We say that $f$ is \emph{not cylindrical over a curve (NCC) over $K$} if there does not exist a $K$-linear map $\ell : \A_K^n \to \A_K^2$ and a curve $C$ in $\A_K^2$ such that $X = \ell^{-1}(C)$. For convenience we will say that any $f$ is NCC if $m \leq 2$.
\end{definition}

Note that $f$ is cylindrical over a curve if and only if there exist a polynomial $g \in K[Y_1,Y_2]$ and linear forms $\ell_1(\bm X)$, $\ell_2(\bm X)$ over $K$ such that $f(\bm X) = g(\ell_1(\bm X), \ell_2(\bm X))$.

\subsection{Statement of results} Firstly, as mentioned above, we can achieve a finiteness result for $N_F^*$, under certain conditions, when we restrict to integral points.

\begin{prop} \label{prop:gcdFinite}
Let $F=(f_1,\ldots,f_n) \in \Z[X_1,\ldots,X_m]^n$ and $S \subseteq \Z^m$ be such that 
\begin{itemize}
\item pairwise the $f_i$ have no common zeros in $\C^m$;
\item for each $1 \leq i \leq n$, $f_i(\bm u)= \pm 1$ has only finitely many solutions $\bm u \in S$; and
\item for all but at most one $1 \leq i \leq n$, $P^+(f_i(\bm u)) \to \infty$ as $\max_{1 \leq i \leq m} |u_i| \to \infty$ for $\bm u = (u_1,\ldots,u_m) \in S$, where $P^+(x)$ denotes the largest prime factor of $x \in \Z$.
\end{itemize}
Then $N_F^*(S)$ is finite.
\end{prop}

The last condition on the growth of the largest prime factor of a polynomial value is satisfied (when $S$ is the set of relatively prime $m$-tuples of integers) by
many classes of polynomials, including for example binary forms with at least three distinct linear factors, discriminant forms and index forms, and a large class of norm forms (see \cite{G}).


Given a multiplicatively dependent vector $\bm{\nu}$ of algebraic numbers, it follows from work of Loxton, van der Poorten \cite{PL} and various others, that there is a relation of the form \eqref{eq:MultDep} with a non-zero vector $\bm{k}$ with small coordinates. In the case where $\bm{\nu}$ consists of rational function values, we can use this together with an elementary argument to obtain the following bound.

\begin{prop} \label{prop:NaiveCount}
Let $F=( f_1,\ldots,f_n ) \in \overline{\Q}(X_1,\ldots,X_m)^n$, let $H \geq 0$, and let $S$ be a finite subset of $\{ x \in \overline{\Q} \mid \mathrm{H}(x) \leq H \}$. Then $N_F(S^m) \ll |S|^{mn-1} (\log H)^{n^2-1}$. 
\end{prop}

Let $K$ be a number field of degree $D$ over $\Q$, with ring of integers $\cO_K$. Let $\cB_K(H)$ denote the set of algebraic integers in $\cO_K$ of height at most $H$, and put $B_K(H) = |\cB_K(H)|$. It follows from the work of Widmer \cite[Theorem~1.1]{W} that
\begin{equation} \label{eq:BKH}
B_K(H) \ll_K H^D (\log H)^r,
\end{equation}
where $r=r_1+r_2-1$, with $r_1$ and $r_2$ the number of real and pairs of complex conjugate embeddings of $K$, respectively. Applying this to Proposition~\ref{prop:NaiveCount}, for any $F \in \overline{\Q}(X_1,\ldots,X_m)^n$, we have
$$
N_F(\cB_K(H)^m) \ll H^{D(mn-1)} (\log H)^{r(mn-1)+n^2-1}.
$$
Combining the approach from \cite{PSSS} with aforementioned dimension growth bounds, we can significantly improve the above estimate in the polynomial case, under certain conditions on $F$ (see Definition~\ref{def:NCC}).

\begin{theorem} \label{thm:AICount}
Let $F = (f_1,\ldots,f_n) \in \cO_K[X_1,\ldots,X_m]^n$. If $m > 1$, assume that for $i=1,\ldots,n$, $f_i-\alpha$ is irreducible and NCC over $K$ for all $\alpha \in \cO_K$, the homogeneous part of $f_i$ of degree $\deg f_i$ is $1$-irreducible (and geometrically irreducible if $m = 2$), and that $d := \min_i \deg f_i \geq 3$. Then there exists $\kappa=\kappa(m)$ (with $\kappa(1)=0$) such that
$$
N_F(\cB_K(H)^m) \ll H^{D(mn-v(m,d))} (\log H)^{rm(n-1)+\kappa+3},
$$
where
\begin{equation} \label{eq:v}
v(m,d) = \begin{cases} 2, & d \geq 4, \: m > 1 \\ 3-\frac{2}{\sqrt{3}}, & d=3, \: m > 1 \\ 1 & m = 1, \end{cases}
\end{equation}
and the implied constant depends only on $F$ and $K$. Moreover, if $m > 2$, $d \geq 5$, and the homogeneous part of each $f_i$ is $2$-irreducible, then the $\kappa$ in the exponent of $\log H$ can be removed.
\end{theorem}

As in \cite{PSSS}, we will compute a bound in the case of a fixed multiplicative rank $0 \leq s \leq n-1$, and note that the majority of the contribution to the bound in Theorem~\ref{thm:AICount} comes from the rank 0 and rank 1 terms in the decomposition \eqref{eq:rankDecomp}, which we treat separately.

\begin{prop} \label{prop:RankCount}
Let $F = (f_1,\ldots,f_n) \in \cO_K[X_1,\ldots,X_m]^n$. If $m > 1$ assume that for $i=1,\ldots,n$, $f_i-\alpha$ is irreducible and NCC over $K$ for all $\alpha \in \cO_K$, the homogeneous part of $f_i$ of degree $\deg f_i$ is $1$-irreducible (and geometrically irreducible if $m=2$), and $d := \min_i \deg f_i \geq 3$.  Then for $0 \leq s < n$ we have
\begin{equation*}
N_{F,s}(\cB_K(H)^m) \ll H^{D(mn-\lceil (s+1)/2 \rceil v(m,d)) + o(1)},
\end{equation*}
where $v(m,d)$ is defined as in \eqref{eq:v}, and the implied constants depend only on $K$ and $F$.
\end{prop}

\subsection{Acknowledgements} The author is very grateful to Igor Shparlinski for suggesting the problems discussed in the paper, as well as giving useful comments on an initial draft.

\section{A finiteness result} \label{sec:finite}

We begin by quickly proving Proposition~\ref{prop:gcdFinite}. The key ingredient in the proof is the fact that for a collection $f_1,\ldots,f_n \in \Z[X_1,\ldots,X_m]$ coprime polynomials, $\gcd(f_1(\bm u), \ldots, f_n(\bm u))$ takes on only finitely many values as $\bm u$ ranges over $\Z^m$. This can be proved directly using the Nullstellensatz, but is also a consequence of work of Bodin and D\`{e}bes.

\begin{prop} \label{prop:BD}
Let $f_1,\ldots,f_n \in \Z[X_1,\ldots,X_m]$ be nonzero polynomials with no common zero in $\C^m$. Then the set $\{ \gcd(f_1(\bm u),\ldots,f_n(\bm u)) : \bm u \in \Z^m \}$ is finite.
\end{prop}

\begin{proof}
This is part of \cite[Corollary~1.4]{BP}.
\end{proof}

We will show that a multiplicative dependence relation leads to a contradiction with Proposition~\ref{prop:BD}, under the assumption that the largest prime factor of values of our polynomials grow with the height of the input. 

\begin{proof}[Proof~of~Proposition~\ref{prop:gcdFinite}]
Suppose $\mathbf{u} = (u_1,\ldots,u_m) \in S$ is such that $f_1(\mathbf{u}), \ldots, f_n(\mathbf{u})$ are multiplicatively dependent, say
$$
\prod_{i=1}^n f_i(\mathbf{u})^{k_i} = 1,
$$
with the $k_i \in \Z$ not all zero. Suppose $f_i(\mathbf{u}) \neq \pm 1$ for each $i$, as this is only satisfied for finitely many $\mathbf{u}$ by assumption, so our multiplicative dependence relation is of rank at least 1. Fix $1 \leq i \leq n$ with $k_i \neq 0$ and such that $P^+(f_i(\bm u)) \to \infty$ as $\max_{1 \leq i \leq m} |u_i| \to \infty$. Then each prime $p$ dividing $f_i(\mathbf{u})$ divides some $f_j(\mathbf{u})$ with $j \neq i$ and $k_j \neq 0$. By Proposition~\ref{prop:BD}, the set of possible values for $\gcd(f_i(\mathbf{u}), f_j(\mathbf{u}))$ is finite, but this contradicts our assumption that $P^+(f_i(\bm u)) \to \infty$ as $\max_{1 \leq i \leq m} |u_i| \to \infty$.
\end{proof}

\section{Statistical results} \label{sec:Stat}

In this section we will prove Theorem~\ref{thm:AICount} and related results. We begin with some preliminaries.

\subsection{Weil height} We record some well-known results about the absolute Weil height \cite[\textsection \textsection 3.2]{Wald}. Let $\alpha_1,\ldots,\alpha_N$ be algebraic numbers. Then
\begin{align} \label{eq:HeightSumBounds}
\mathrm{H}(\alpha_1 \cdots \alpha_N) & \leq \mathrm{H}(\alpha_1) \cdots \mathrm{H}(\alpha_N) \notag \\
\mathrm{H}(\alpha_1 + \cdots + \alpha_N) & \leq N \mathrm{H}(\alpha_1) \cdots \mathrm{H}(\alpha_n),
\end{align}
and if $\alpha_1 \neq 0$ and $k$ is an integer, then
\begin{equation} \label{eq:HeightPower}
\mathrm{H}(\alpha_1^k)=\mathrm{H}(\alpha_1)^{|k|}.
\end{equation}

We also recall a well-known comparison between the naive height $\mathrm{H}_0$ and the absolute Weil height $\mathrm{H}$ \cite[Equation~(6)]{Mah}. Let $\alpha$ be an algebraic number of degree $D$. Then
\begin{equation} \label{eq:NaiveHeightCompare}
\mathrm{H}_0(\alpha) \leq (2 \mathrm{H}(\alpha))^D.
\end{equation}



Moreover, we note the following simple bound on the height of a polynomial value. 

\begin{lemma} \label{lem:HeightTransform}
Let $f \in \overline \Q[X_1,\ldots,X_m]$, and let $\bm u \in \overline \Q^m$ have coordinates of height at most $H$. Then there exists a constant $C_f > 0$, depending only on $f$, such that $\mathrm{H}(f(\bm{u})) \leq H^C$.
\end{lemma}

\begin{proof}
Write
$$
f(X_1,\ldots,X_m) = \sum_{k=1}^\ell a_k X_1^{\lambda_{k1}} \cdots X_m^{\lambda_{km}},
$$
and $\bm{u}=(u_1,\ldots, u_m)$ with $\mathrm{H}(u_i) \leq H$ for $1 \leq i \leq m$. Then by \eqref{eq:HeightSumBounds},
\begin{align*}
\mathrm{H} \left( f(\bm u) \right) & = \mathrm{H} \left( \sum_{k=1}^\ell a_k u_1^{\lambda_{k1}} \cdots u_m^{\lambda_{km}} \right) \\
& \leq \ell \prod_{k=1}^\ell \left( \mathrm{H}\left(a_k\right) \prod_{i=1}^m \mathrm{H} \left(u_i \right)^{\lambda_{ki}} \right) \leq H^{C_f},
\end{align*}
where
$$
C_f = \log_H \left( \ell \prod_{k=1}^\ell \mathrm{H}(a_k) \right) + \sum_{k=1}^\ell \sum_{i=1}^m \lambda_{kj}
$$
depends only on $f$.
\end{proof}

\subsection{Counting integral solutions to polynomial equations} We now state the affine dimension growth bounds which will be useful for our purposes. These follow from \cite[Theorem~4.1]{CDHNV} and \cite[Theorem~5]{BCK}. Recall that NCC is defined in Definition~\ref{def:NCC}.

\begin{theorem} \label{thm:DimGrowth}
Let $K$ be a number field of degree $D$. Given $m > 1$ there exist a constant $\kappa=\kappa(m)$ such that for all polynomials $f \in \cO_K[X_1,\cdots,X_m]$ of degree $d \geq 3$ such that $f$ is irreducible over $K$ and NCC, the following holds. If the homogeneous part of $f$ of degree $d$ is $1$-irreducible over $K$ (and geometrically irreducible if $m=2$), we have
$$
\left| \{ \bm u \in \cB_K(H)^m : f(\bm u) = 0 \} \right| \ll_{K,m} d^7 H^{D(m-v(m,d))} (\log H)^\kappa, 
$$
where $v(m,d)$ is defined as in \eqref{eq:v}. In the case $m=2$ we can replace $d^7$ by $d^2$.

If moreover $m > 2$, $d \geq 5$ and the homogeneous part of $f$ of degree $d$ is $2$-irreducible over $K$, then
$$
\left| \{ \bm u \in \cB_K(H)^m : f(\bm u) = 0 \} \right| \ll_{K,m} d^7 H^{D(m-2)}.
$$
\end{theorem}

\subsection{Multiplicative structure of algebraic numbers} The next result shows that given multiplicatively dependent algebraic numbers, we can find a relation as \eqref{eq:MultDep} where the exponents are not too large.

\begin{lemma} \label{lem:Siegel}
Let $n \geq 2$, and let $\alpha_1,\ldots,\alpha_n$ be multiplicatively dependent non-zero algebraic numbers of degree at most $D$ and height at most $H$. Then there exists a positive number $A$, depending only on $n$ and $D$, and there are rational integers $k_1,\ldots,k_n$, not all zero, such that
$$
\alpha_1^{k_1} \cdots \alpha_n^{k_n} = 1
$$
and
$$
\max_{1 \leq i \leq n} |k_i| < A(\log H)^{n-1}.
$$
\end{lemma}

\begin{proof}
This follows from \cite[Theorem~1]{PL}.
\end{proof}

For positive integers $x$ and $y$, let $\psi(x,y)$ denote the number of positive integers not exceeding $x$ which have no prime factors greater than $y$. We have the following bound \cite[Theorem~1]{Bru}.

\begin{lemma} \label{lem:SmoothBound}
Let $2 < y \leq x$ be integers. Then
\begin{align*}
\psi(x, & y) \\
& = \exp \left(Z \left((1+O((\log y)^{-1}) + O((\log \log x)^{-1}) + O((u+1)^{-1}) \right) \right),
\end{align*}
where
$$
Z = \left( \log \left(1+\frac{y}{\log x} \right) \right) \frac{\log x}{\log y} + \left( \log \left(1+ \frac{\log x}{y} \right) \right) \frac{y}{\log y}
$$
and
$$
u = \frac{\log x}{\log y}.
$$
\end{lemma}

We also note a bound on algebraic numbers whose minimal polynomials have fixed leading and constant coefficients \cite[Lemma~2.5]{PSSS}.

\begin{lemma} \label{lem:leadCoeffCount}
Let $K$ be a number field of degree $D$, and let $u$ and $v$ be non-zero integers with $u > 0$. Then there is a positive number $c$, which depends on $K$, such that the number of elements $\alpha$ in $K$ of height at most $H$, whose minimal polynomial has leading coefficient $u$ and constant coefficient $v$, is at most $\exp(c \log H/ \log \log H)$.
\end{lemma}

\subsection{Proofs of Propositions \ref{prop:NaiveCount} and \ref{prop:RankCount}, and Theorem~\ref{thm:AICount}} In light of Lemma~\ref{lem:Siegel}, the proof of Proposition~\ref{prop:NaiveCount} is fairly straightforward.

\begin{proof}[Proof of Proposition~\ref{prop:NaiveCount}]
Write $f_i=g_i/h_i$, $g_i,h_i \in \overline{\Q}[X_1,\ldots,X_m]$, $i=1,\ldots,n$, and let $(\bm u_1, \ldots, \bm u_n) \in N_F(S^m)$. Then, where $\bm u_i = (u_{i1},\ldots,u_{im})$, we have $\mathrm{H}(u_{ij}) \leq H$, $i=1,\ldots,n$, $j=1,\ldots,m$. Using Lemma~\ref{lem:HeightTransform} and \eqref{eq:HeightPower}, for $i=1,\ldots,n$ we have
$$
\mathrm{H} \left( f_i(\bm u_i) \right) \leq 2 \mathrm{H} \left( g_i(\bm u_i) \right) \mathrm{H} \left( h_i(\bm u_i) \right) \leq H^C,
$$
where $C > 0$ is a constant depending only on $f_1,\ldots,f_n$. Therefore, by Lemma~\ref{lem:Siegel}, there exists a constant $A > 0$ depending only on $f_1,\ldots,f_n$, and integers $k_1,\ldots,k_n$, not all zero and with $\max_{1 \leq i \leq n} |k_i| \leq M := AC^{n-1} \left( \log H \right)^{n-1}$, such that 
$$
\prod_{i=1}^n f_i(\bm u_i)^{k_i} = 1,
$$
or equivalently
\begin{equation} \label{eq:MDRel}
\prod_{k_i > 0} g_i(\bm u_i)^{k_i} \prod_{k_i < 0} h_i(\bm u_i)^{k_i} - \prod_{k_i < 0} g_i(\bm u_i)^{k_i} \prod_{k_i > 0} h_i(\bm u_i)^{k_i} = 0.
\end{equation}
Fix a vector $(k_1,\ldots,k_n) \in [-M,M]^n$, then view \eqref{eq:MDRel} as a polynomial equation of degree at most $2M \left( \sum_{i=1}^n \deg f_i \right)$ in the $mn$ variables $u_{ij}$, $i=1,\ldots,n$, $j=1,\ldots,m$. Letting all but one of the $u_{ij}$ run freely over $S$, we see that this has $O(M|S|^{mn-1})$ solutions. Therefore
$$
N_F(S^m) \ll M^{n+1} |S|^{mn-1} \ll |S|^{mn-1} \left( \log H \right)^{n^2-1},
$$
as desired.
\end{proof}

We now proceed to prove our bounds for fixed multiplicative rank.

\begin{proof}[Proof of Proposition \ref{prop:RankCount}]
We fix an $n$-tuple of polynomials $F=( f_1,\ldots,f_n )$ in $\cO_K[X_1,\ldots,X_m]^n$, and let $c_1,c_2,\ldots$ denote positive numbers depending on $F$ and $K$. Suppose $\bm u = (\bm u_1,\ldots,\bm {u}_n) \in (\cO_K^m)^n$ has coordinates of height at most $H$, and is such that $(f_1(\bm{u}_1),\ldots, f_n(\bm{u}_n))$ is a multiplicatively dependent vector of rank $s$. Then there exist $s+1$ distinct integers $\ell_1, \ldots, \ell_{s+1} \in \{1, \ldots, n \}$, and non-zero integers $k_{\ell_1}, \ldots, k_{\ell_{s+1}}$ such that
$$
\prod_{i=1}^{s+1} f_{\ell_i}(\bm u_{\ell_i})^{k_{\ell_i}} = 1.
$$
By Lemma~\ref{lem:HeightTransform}, there is a constant $C > 0$ depending only on $f_1,\ldots,f_n$ such that $H(f_{\ell_i}(\bm u_{\ell_i})) \leq H^C$ for $1 \leq i \leq s+1$, and so by Lemma~\ref{lem:Siegel}, we can assume that
\begin{equation} \label{eq:RankSiegel}
\max \{ |k_{\ell_1}|, \ldots, |k_{\ell_{s+1}}| \} < c_1 (\log H)^s.
\end{equation}
Let $I,J \subseteq \{\ell_1, \ldots, \ell_{s+1}\}$ be respectively the sets of indices satisfying $k_{\ell_i} > 0$ and $k_{\ell_i} < 0$. Then we have the disjoint union $I \sqcup J = \{\ell_1,\ldots,\ell_{s+1} \}$, and
\begin{equation} \label{eq:PNDecomp1}
\prod_{i \in I} f_i(\bm u_i)^{|k_i|} = \prod_{j \in J} f_j(\bm u_j)^{|k_j|}.
\end{equation}
Either $I$ or $J$ has size at least $\lceil (s+1)/2 \rceil$, so assume $|I| \geq \lceil (s+1)/2 \rceil$. 
We will fix $\bm u_j \in \cB_K(H)^m$ for $j \in J$ and estimate the number of solutions of \eqref{eq:PNDecomp1} in $\bm u_i \in \cB_K(H)^m$, $i \in I$. Note that by \eqref{eq:RankSiegel}, the number of cases when we consider an equation of the form \eqref{eq:PNDecomp1} is at most
$$
\binom{n}{s+1} \left( 2c_1 (\log H)^s \right)^{s+1} B_K(H)^{m(n-|I|)},
$$
and by \eqref{eq:BKH}, this is at most
\begin{equation}
\label{eq:RkAIBd1}
c_2 H^{Dm(n-|I|)} (\log H)^{c_3}.
\end{equation}
Let $q_1,\ldots,q_t$ be the primes which divide
$$
\prod_{j \in J} N_{K/\Q}(f_j(\bm u_j)),
$$
where $N_{K/\Q}$ is the norm from $K$ to $\Q$. Since the height of $f_j(\bm u_j)$ is at most $H^C$, where $C>0$ is as we defined it before the equation \eqref{eq:RankSiegel}, we see from \eqref{eq:NaiveHeightCompare} that
\begin{equation} \label{eq:ProdNormBd}
\left| \prod_{j \in J} N_{K/\Q}(f_j(\bm u_j)) \right| \leq (2H^C)^{Dn},
\end{equation}
noting that $|J| \leq n$. Let $p_1, \ldots, p_k$ be the first $k$ primes, where $k$ satisfies
$$
p_1 \cdots p_t \leq \left| \prod_{j \in J} N_{K/\Q}(f_j(\bm u_j)) \right| < p_1 \cdots p_{k+1}.
$$
Let $T$ denote the number of positive integers up to $(2H^C)^{D}$, composed only of primes from $\{q_1, \ldots q_t \}$. Then $T$ is bounded above by the number of positive integers up to $(2H^C)^{D}$ composed of primes from $\{p_1,\ldots,p_k\}$. Then from \eqref{eq:ProdNormBd},
$$
\sum_{\text{prime } p \leq p_k} \log p \ll \log H,
$$
which, combined with the prime number theorem, gives
$$
p_k < c_4 \log H.
$$
Thus $T \leq \psi \left( (2H^C)^D, c_4 \log H \right)$, and so by Lemma~\ref{lem:SmoothBound},
\begin{equation} \label{eq:TBd}
T < \exp(c_5 \log H/ \log \log H).
\end{equation}
Therefore, if $\bm u_i$, $i \in I$ give a solution to \eqref{eq:PNDecomp1}, then $|N_{K/\Q}(f_i(\bm u_i))|$ is composed only of primes from $\{ q_1, \ldots, q_t \}$, and so $N_{K/\Q}(f_i(\bm u_i))$ is one of at most $2T$ integers of absolute value at most $(2H^C)^D$. Let $a$ be one such integer.

By Lemma~\ref{lem:leadCoeffCount}, the number of $\alpha \in \cB_K(H)$ for which $N_{K/\Q}(\alpha)=a$ is at most $\exp(c_6 \log H/ \log \log H)$. Given such an $\alpha$, the number of $\bm u_i \in \cB_K(H)^m$ such that $f_i(\bm u_i) = \alpha$ is $\ll H^{D(m-v(m,d))} (\log H)^\kappa$ by Theorem~\ref{thm:DimGrowth}, where $d := \min_{1 \leq i \leq n} \deg f_i$ and $v(m,d)$ is as defined in \eqref{eq:v}. Note that this holds trivially with $\kappa = 0$ if $m=1$.

Therefore, by \eqref{eq:TBd}, the number of $|I|$-tuples $(\bm u_i, i \in I)$ which give a solution to \eqref{eq:PNDecomp1} is at most
$$
H^{D|I|(m-v(m,d))} \exp( c_7 \log H/\log \log H).
$$
Recalling that $|I| \geq \lceil (s+1)/2 \rceil$, we see that
$$
N_{F,s}(\cB_K(H)^m) < H^{D(mn-\lceil (s+1)/2 \rceil v(m,d))} \exp( c_8 \log H/\log \log H),
$$
completing the proof.
\end{proof}

We conclude the paper with the proof of Theorem~\ref{thm:AICount}.

\begin{proof}[Proof of Theorem~\ref{thm:AICount}]
Subsequently, when using the symbol $\ll$, the implied constant will depend on $F$ and $K$. By Proposition~\ref{prop:RankCount}, we see that
$$
\sum_{s=2}^{n-1} N_{F,s}(\cB_K(H)^m) \ll H^{D(mn-2v(m,d))+o(1)},
$$
so we focus on controlling $N_{F,0}(\cB_K(H)^m)$ and $N_{F,1}(\cB_K(H)^m)$. In the rank 0 case, we can fix some index $i$ for which $f_i(\bm u_i)$ is a root of unity in $K$, and choose $\bm u_j$ arbitrarily in $\cB_K(H)^m$ for $j \neq i$. This can be done in $\ll H^{Dm(n-1)} (\log H)^{rm(n-1)}$ ways by \eqref{eq:BKH}. By Theorem~\ref{thm:DimGrowth}, for each of the $\ll 1$ roots of unity $\zeta \in \cO_K$, there are $\ll H^{D(m-v(m,d))} (\log H)^\kappa$ choices $\bm u_i \in \cB_K(H)^m$ for which $f_i(\bm u_i)=\zeta$ (again the $m=1$ case is trivial). We conclude that 
$$
N_{F,0}(\cB_K(H)^m) \ll H^{D(mn-v(m,d))} (\log H)^{rm(n-1)+\kappa}.
$$
Now, suppose $(\bm u_1, \ldots \bm u_n) \in N_{F,1}(\cB_K(H))$. Then by Lemma~\ref{lem:Siegel}, there exist $i \neq j$ such that $f_i(\bm u_i)^{k_1} = f_j(\bm u_j)^{k_j}$ for integers $k_i,k_j$, not both zero, with $|k_i|,|k_j| \ll \log H$. Without loss of generality, assume $k_i > 0$. Fix $i$, $j$, $k_j$ and $\bm u_\ell$ ($\ell \neq i$), for which there are $\ll H^{Dm(n-1)} (\log H)^{rm(n-1)+1}$ possibilities by \eqref{eq:BKH}. Then $f_i(\bm u_i) = \alpha$, where $\alpha$ is one of at most $k_i$ roots of $X^{k_i}-f_j(\bm u_j)^{k_j}$ in $\cO_K$. For each such $\alpha$, there are $\ll H^{D(m-v(m,d))} (\log H)^\kappa$ possibilities for $\bm u_i$ by Theorem~\ref{thm:DimGrowth}. Summing over $|k_i| \ll \log H$ gives $\ll H^{D(m-v(m,d))} (\log H)^{\kappa+2}$ total possibilities for $\bm u_i$. We conclude that
$$
N_{F,1}(\cB_K(H)^m) \ll H^{D(mn-v(m,d))} (\log H)^{rm(n-1)+\kappa+3},
$$
and the result follows, noting again by Theorem~\ref{thm:DimGrowth}, that if additionally $m>2$, $d \geq 5$ and the top homogeneous component of each $f_i$ is $2$-irreducible, then all instances of $\kappa$ can be removed.
\end{proof}

\Address

\end{document}